\documentclass{article}
\usepackage{amssymb}

\newtheorem{theorem}{Theorem}[section]

\newtheorem{example}{Example}[section]

\newtheorem{problem}{Problem}[section]
\newtheorem{proposition}{Proposition}[section]

\newenvironment{proof}[1][Proof]{\noindent\textbf{#1.} }{\ \rule{0.5em}{0.5em}}
\input{tcilatex}
\begin{document}

\author{Faruk Abi-Khuzam \\
%EndAName
Department of Mathematics, American University of Beirut,\\
Beirut, Lebanon}
\title{Elementary Discrete and Continuous Interplay}
\maketitle

\begin{abstract}
We illustrate the interplay between certain discrete and continuous
problems, by presenting a method for the study of the asymptotics of a
divergent sequence, through consideration of the asymptotics of its
continuous analogue.
\end{abstract}

\section{\protect\bigskip Introduction}

The alternating harmonic series $\sum_{k=1}^{\infty }\frac{(-1)^{k+1}}{k}$
is convergent and its sum is $\ln 2$. The non-linear difference equation%
\[
a_{n+1}=a_{n}+\frac{1}{a_{n}},a_{1}=1, 
\]%
has a divergent solution which grows like $\sqrt{2n}$\ $\ $as $n\rightarrow
\infty $. The first fact is common knowledge usually proved by going to the
Maclaurin expansion of $\ln (1+x)$ and then laboring a bit to show that it
remains valid for $x=1$. The second fact would probably strike most readers
as nothing more than a curiosity, perhaps obtained by some clever
manipulation. Both problems belong in a discrete setting, though their
answers derive somehow from the continuous setting. But this is no strange
thing, since the continuous is constructed from the discrete through the
completeness property. What is not clear is that these seemingly disparate
mathematical facts, can be viewed on a common platform, and their solutions
can be obtained from a method that exploits what we shall call here the
discrete-continuous interplay. The method, which is the subject of this
article, is elementary, requiring no prerequisite beyond calculus and a
little knowledge of differential equations. It can be used in any course
where the theory of convergence is given, and allows both instructor and
student to go beyond the dichotomy of convergence-divergence of a given
sequence or series. Specifically, it will make it possible to enrich the
study of convergence theory, by opening the possibility of treating novel,
and useful, problems by a method that, from the start, suggests a line of
attack and a possible answer. The idea is that a problem in the discrete
setting is translated, by use of a simple dictionary, into a problem in the
continuous setting, where a solution is sought. If a solution of the
continuous problem is found, it will suggest an answer and a possible
approach for the solution of the discrete problem.

\section{\protect\bigskip The integral test}

One of the earliest exposures to the interplay between the discrete and the
continuous occurs in the integral test for convergence of series, familiar
to all beginners in calculus. The integral test is based on the two
inequalities%
\[
\int_{1}^{n+1}f(t)dt<f(1)+f(2)+\cdot \cdot \cdot
+f(n)<f(1)+\int_{1}^{n}f(t)dt, 
\]%
which are valid for any function $f$ defined on $[1,\infty ),$ positive, and
non-increasing there. The middle sum is the $n^{th}$ partial sum $S_{n}$ of
the series $\sum_{k=1}^{\infty }f(k)$, while the integral on the right is
the "continuous analogue" of $S_{n}$. These inequalities imply that the
infinite series $\sum_{k=1}^{\infty }f(k)$, converges if and only if the
improper integral $\int_{1}^{\infty }f(t)dt$ converges. But as we shall see
readily, these inequalities supply one of many bridges between the discrete
and continuous. As a warm up for what is to come, we shall show how it can
be employed to find the sum of a convergent series from the divergence of
another. To illustrate, consider the case where $f(t)=\frac{1}{t}$. Then we
have 
\[
\int_{1}^{n+1}\frac{1}{t}dt<1+\frac{1}{2}+\cdot \cdot \cdot +\frac{1}{n}%
<1+\int_{1}^{n}\frac{1}{t}dt, 
\]%
and we conclude that the harmonic series $\sum_{k=1}^{\infty }\frac{1}{k}$
is divergent. But it is crucial not to be content with this conclusion.
Indeed, if we introduce the sequence $H_{n}$ defined by 
\[
H_{n}=1+\frac{1}{2}+\cdot \cdot \cdot +\frac{1}{n}-\int_{1}^{n}\frac{1}{t}%
dt, 
\]%
then $H_{n}$ is a positive, monotone decreasing sequence, and hence
convergent. Its limit is usually denoted by $\gamma $, and called the
Euler-Mascheroni constant. \ Having this new information at hand, and
noticing that 
\[
\sum_{k=1}^{2n}\frac{(-1)^{k+1}}{k}=H_{2n}-H_{n}+\int_{1}^{2}\frac{1}{t}dt, 
\]%
we are led immediately, and without any appeal to a Maclaurin series, to the
sum of the alternating series 
\[
\sum_{k=1}^{\infty }\frac{(-1)^{k+1}}{k}=\int_{1}^{2}\frac{1}{t}dt=:\ln 2. 
\]%
So the divergence of the harmonic series, as seen through the two
inequalities between the discrete and the continuous, led to a convergent
sequence, which in turn led to the convergence and evaluation of the sum of
another infinite series. The value of $\ln 2$ for the sum of the series,
turned out to be independent of the limit of the sequence that helped find
it.

The point made in the previous example deserves another illustration, so let
us consider the series $\sum_{k=1}^{\infty }\frac{(-1)^{k+1}\ln k}{k}$. \
The relevant function here is $\frac{\ln t}{t}$ which is decreasing on $%
[e,\infty )$, so that the sequence $T_{n}$ defined by 
\[
T_{n}=\frac{\ln 1}{1}+\frac{\ln 2}{2}+\cdot \cdot \cdot +\frac{\ln n}{n}-%
\frac{1}{2}\ln ^{2}n 
\]%
is convergent. Once again we compute that 
\[
\sum_{k=1}^{2n}\frac{(-1)^{k+1}\ln k}{k}=\sum_{k=1}^{2n}\frac{\ln k}{k}%
-\sum_{k=1}^{n}\frac{\ln 2k}{k} 
\]%
\[
=T_{2n}+\frac{1}{2}\ln ^{2}2n-\ln 2\left( H_{n}+\ln n\right) -T_{n}-\frac{1}{%
2}\ln ^{2}n, 
\]%
and it follows, painlessly, that 
\[
\sum_{k=1}^{\infty }\frac{(-1)^{k+1}\ln k}{k}=\frac{1}{2}\ln ^{2}2-\gamma
\ln 2. 
\]%
Observe that the sum depends on the limit of the sequence $H_{n}$, but not
on that of $T_{n}$. \ The reader can now, if he so wishes, evaluate the sums 
$\sum_{k=1}^{\infty }\frac{(-1)^{k+1}\ln ^{p}k}{k}$, where $p$ is a
non-negative integer. He will be led to the introduction of constants $%
\gamma _{\alpha }$ defined as limits%
\[
\gamma _{\alpha }=\lim_{n\rightarrow \infty }\frac{\ln ^{\alpha }1}{1}+\frac{%
\ln ^{\alpha }2}{2}+\cdot \cdot \cdot +\frac{\ln ^{\alpha }n}{n}-\frac{1}{%
\alpha +1}\ln ^{\alpha +1}n, 
\]%
and will be gratified to know they are called Stieltjes constants, appear as
coefficients in the Laurent expansion of the zeta function of Riemann about
the pole $1$, and continue to be the subject of research \cite{Bla}. But
they will not be considered any further in this article.

It should be clear that the previous discussion can be summed up in the
following simple proposition whose proof is left to the reader.

\begin{proposition}
Let $f$ be a positive function defined on $[1,\infty )$, and monotone
non-increasing on $(a,\infty )$, where $a\geq 1$. Then the sequences 
\[
A_{n}=\sum_{k=1}^{n}f(k)-\int_{1}^{n}f(t)dt,B_{n}=\sum_{k=1}^{n}f(2k)-%
\int_{1}^{n}f(2t)dt 
\]%
are convergent. If , in addition, the series $\sum_{k=1}^{\infty }f(k)$ is
divergent, and $\lim_{k\rightarrow \infty }f(k)=0$, then the series $%
\sum_{k=1}^{\infty }(-1)^{k+1}f(k)$ is convergent, and its sum is 
\[
\sum_{k=1}^{\infty }(-1)^{k+1}f(k)=L+\int_{1}^{2}f(t)dt. 
\]%
where $L=\lim_{n\rightarrow \infty }(A_{2n}-2B_{n})=\lim_{n\rightarrow
\infty }(A_{n}-2B_{n}).$
\end{proposition}

\section{From discrete to continuous}

The integral $\int_{1}^{n}f(t)dt$ is here proposed as the "continuous
analogue" of the "discrete" sum $\sum_{k=1}^{n}f(k)$, but the two are not
equal. However, see \cite{BoPo} for the case $n=\infty .$In addition, we
propose, the derivative $f^{\prime }(k)$ as the continuous analogue of the
difference $f(k+1)-f(k)$. The first step in the method that we propose here
is to make a translation of the problem in the discrete setting into a
continuous problem, using the dictionary as in the following table:%
\[
\begin{array}{ccc}
Item & Discrete & Continuous \\ 
variable & n & t \\ 
function & a_{n} & f(t) \\ 
derivative & a_{n+1}-a_{n} & f^{\prime }(t) \\ 
integral & \sum_{k=1}^{n}a_{k} & \int_{1}^{n}f(t)dt%
\end{array}%
\cdot 
\]

No claim is made as to how "faithful" this translation would be. For
example, we know that $a_{n}$ $\rightarrow 0$, if $\sum_{k=1}^{n}a_{k}$
converges, but $f$ may not have a limit at $\infty $ if the integral $%
\int_{1}^{\infty }f(t)dt$ converges. If $a_{n}\rightarrow a$, as $%
n\rightarrow \infty $, then the difference $a_{n+1}-a_{n}\rightarrow 0$, but
this is not true of a function and its derivative. Neverthless, use of this
dictionary back and forth, will invariably lead to interesting novel
problems. For example, it is known that for a function $f$, differentiable
on $(a,\infty )$, the condition 
\[
\alpha >0,f^{\prime }(t)+\alpha f(t)\rightarrow 0,t\rightarrow \infty , 
\]%
implies that $f(t)\rightarrow 0$, as $t\rightarrow \infty .$ A translation
of this gives the following problem on a real sequence $a_{n}$: suppose that 
\[
a_{n+1}-a_{n}+\alpha a_{n}=a_{n+1}+\beta a_{n}\rightarrow 0,n\rightarrow
\infty , 
\]%
does it follow that $a_{n}\rightarrow 0$, as $n\rightarrow \infty ?$

Once a translation is made from the discrete to the continuous, the task
shifts to a search for a solution of the continuous problem, which might
involve a simple differential equation with a readily computable solution,
or a very difficult integro-differential equation, or just a brief limit
statement . But in solving any one of these, if that is possible, due
attention must be made to the steps taken, in order to be able to translate
them, usually in reverse order, back into the discrete case. Thus when $%
f^{\prime }$ is integrated to get $f$, the corresponding step on the
sequence would be to sum the differences $a_{n+1}-a_{n}$ .This means that
some , sometimes not so simple, maneuvering must be used in order to obtain
a proof in the discrete case. So no royal road is being paved here. The most
important aspect of the method remains in its ability to suggest a possible
line of attack and a possible answer. In addition, its range of
applicability is reasonably wide. We shall illustrate various aspects of
this method throughout the rest of this article, supplying enough details to
ellucidate its usefulness, as well as the subtlties involved in the process.

\ Consider the result stated in the introduction about the asymptotics of
the sequence $a_{n}$ defined by 
\[
a_{n+1}=a_{n}+\frac{1}{a_{n}},a_{1}=1. 
\]%
This is an example of a first order difference equation. Were it linear, it
would have been possible to solve for $a_{n}$ explicitly. But it is a
non-linear difference equation, and so,instead, we search for the asymptotic
behaviour of $a_{n}$ as $n\rightarrow \infty .$ This will give us the order
of growth of $a_{n}$ . It will tell us whether $a_{n}$ grows like some power 
$n^{\alpha }$, or like $\ln n,$ etc... Using the dictionary, with $%
a_{n}=f(t),a_{n+1}-a_{n}=f^{\prime }(t)$, and ignoring for the moment the
initial condition , the corresponding continuous problem involves a positive
function satisfying the differential equation 
\[
f^{\prime }(t)=\frac{1}{f(t)}, 
\]%
which is easily solved to give $f^{2}(t)=2t$ as one particular solution. So
we guess, from the "continuous" answer, that the sequence $a_{n}$ possibly
satisfies $a_{n}\sim \sqrt{2n}$ as $n\rightarrow \infty .$ We also see, from
the continuous solution, that we should look at $a_{n}^{2}.$ With these two
insights at hand we can proceed to a proof. Squaring we first obtain%
\[
a_{n+1}^{2}-a_{n}^{2}=2+\frac{1}{a_{n}^{2}}\geq 2,a_{2}^{2}-a_{1}^{2}=3. 
\]%
Summing the first inequalities from $2$ to $n$, gives us the first key
inequality $a_{n}^{2}\geq 2n$, for $n\geq 2$. But now this inequality gives
us that 
\[
a_{n+1}^{2}-a_{n}^{2}=2+\frac{1}{a_{n}^{2}}\leq 2+\frac{1}{2n},n\geq 2, 
\]%
and again summation gives us the second key inequality%
\[
a_{n+1}^{2}\leq 2(n+1)+\frac{1}{2}\ln n,n\geq 2\text{.} 
\]%
Putting the two together we obtain that

\[
\lim_{n\rightarrow \infty }\frac{a_{n}^{2}}{2n}=1,\lim_{n\rightarrow \infty }%
\frac{a_{n}}{\sqrt{2n}}=1. 
\]%
If instead of the term $\frac{1}{a_{n}}$ we had $\frac{1}{a_{n}^{2}}$, we
would get $f^{3}(t)=3t$, guess that $a_{n}\sim \sqrt[3]{3n}$, and start our
proof by looking at $a_{n+1}^{3}-a_{n}^{3}$. In fact we can replace the term 
$\frac{1}{a_{n}}$ by $\frac{1}{f(a_{n})}$ thereby obtaining the following

\begin{theorem}
Let $f$ be a positive non-decreasing continuous function defined on $%
[0,\infty )$. Let $a_{n}$ be a solution of 
\[
a_{n+1}-a_{n}=\frac{1}{f(a_{n})} 
\]%
satisfying $a_{1}>0$. Put 
\[
F(x)=1+\int_{0}^{x}f(t)dt,\;x\geq 0. 
\]%
Then 
\[
a_{n}\;\sim \;F^{-1}(n)\text{ \ as }n\rightarrow \infty . 
\]
\end{theorem}

\begin{proof}
The sequence $\{a_{n}\}$ is positive and increasing. Since $f$ is
non-decreasing, 
\[
F(a_{n+1})-F(a_{n})=\int_{a_{n}}^{a_{n+1}}f(t)dt\geq
f(a_{n})(a_{n+1}-a_{n})=1, 
\]%
for $n=1,2,3,...\;$. Summing from $1$ to $n$ and noting that $F(a_{1})>1$,
we obtain $F(a_{n+1})\geq n+1$ and so $a_{n}\geq F^{-1}(n)$, since the
inverse function $F^{-1}$ is well-defined. But then $f(a_{n})\geq f\circ
F^{-1}(n)$ because $f$ is non-decreasing and this may be incorporated in the
original difference equation to give 
\[
a_{n+1}-a_{n}\leq \frac{1}{g(n)}\text{ where }g=f\circ F^{-1},\text{ and }%
n=1,2,3,...\;. 
\]%
Once again, if we sum from $1$ to $n$ we obtain 
\[
a_{n+1}-a_{1}\leq \sum_{k=1}^{n}\frac{1}{g(k)}\cdot 
\]%
The composite function $g=f\circ F^{-1}$ is positive,non-decreasing, and
continuous, and $F^{-1}$ is differentiable, so the change of variable
formula \cite{Rudin} \ may be used to get the inequality 
\[
\sum_{k=1}^{n}\frac{1}{g(k)}\leq \frac{1}{g(1)}+\int_{1}^{n}\frac{dx}{g(x)}=%
\frac{1}{g(1)}+F^{-1}(n)-F^{-1}(1). 
\]%
We have $g(1)=f(F^{-1}(1))=f(0)>0$, and so 
\[
a_{n+1}\leq a_{1}+\frac{1}{g(1)}+F^{-1}(n)-F^{-1}(1)\leq c+F^{-1}(n+1) 
\]%
for a constant $c$ independent of $n$.We have thus shown that $F^{-1}(n)\leq
a_{n}\leq F^{-1}(n)+c$, and the result follows since $F^{-1}(n)$ $%
\rightarrow \infty $ as $n\rightarrow \infty $.
\end{proof}

\begin{example}
If $a_{n+1}-a_{n}=a_{n}^{-\alpha },$ where $a_{1},\alpha >0$, then $%
a_{n}\;\sim \;(\alpha +1)^{\frac{1}{\alpha +1}}n^{\frac{1}{\alpha +1}}\,$ as 
$n\rightarrow \infty .$
\end{example}

\begin{example}
If $a_{n+1}-a_{n}=\exp (-a_{n}),$ where $a_{1}>0,$ then $a_{n}\;\sim \;\ln n$
as $n\rightarrow \infty .$
\end{example}

\section{Extensions}

Returning to the solved example of the previous section, if instead of the
term \ $\frac{1}{a_{n}}$ we had $\frac{1}{2na_{n}}$, we would get $%
f^{2}(t)=\ln t,$ guess that $a_{n}\sim \sqrt{\ln n}$ , and again start from
the square of the given sequence. Thus we see that various generalizations
of Theorem $1$ are readily obtained. Here is a simple one whose proof is
left to the reader.

\textit{Let }$f$ and $g$ \textit{be positive, non-decreasing, and continuous
functions defined on }$[0,\infty )$.\textit{\ Let }$a_{n}$\textit{\ be a
solution of }%
\[
a_{n+1}-a_{n}=\frac{1}{f(a_{n})g(n)} 
\]%
\textit{satisfying }$a_{1}>0$\textit{. Put }%
\[
F(x)=1+\int_{0}^{x}f(t)dt,\;x\geq 0, 
\]%
\textit{and assume that }$\int_{1}^{n}\frac{1}{g(t)}dt\rightarrow \infty $%
\textit{\ as }$n\rightarrow \infty .$\textit{Then }%
\[
a_{n}\;\sim \;F^{-1}(\int_{1}^{n}\frac{1}{g(t)}dt)\text{ \ as }n\rightarrow
\infty . 
\]%
We emphasize that our aim here is not complete generality. In fact it is
advisable to treat a given problem on its own merits, as consideration of
the following example, where $g$ is actually decreasing, will reveal.

\begin{example}
If the sequence $a_{n}$ is defined by the recurrence 
\[
a_{n+1}-a_{n}=\frac{n^{\alpha }}{3a_{n}^{2}},a_{1}>0,\alpha >0, 
\]%
then%
\[
a_{n}\sim \frac{1}{\sqrt[3]{\alpha +1}}n^{\frac{\alpha +1}{3}},n\rightarrow
\infty . 
\]
\end{example}

\section{Second order differences}

The results in the previous section involved a sequence and its first
differences. In other words we considered the problem of obtaining the
asymptotic behaviour of a sequence satisfying a first order non-linear
difference equation. Since the equation was not linear, it was not possible
to solve it explicitly and that led to the question of obtaining its
asymptotic behaviour. It is natural to ask whether the method could be of
help in the case where second order differences arise. Of course, there will
be added difficulties. The next example, with all its details, illustrates
both the utility of the method and the difficulties that are bound to arise
as we try to obtain an argument in the discrete setting from the usually
easier argument in the continuous setting.

\begin{problem}
If $a_{n}$ is a sequence satisfying 
\[
a_{n+1}=a_{n}+\frac{1}{a_{n}}\sum_{k=1}^{n}a_{k},a_{1}=1, 
\]%
obtain the asymptotic behaviour of $a_{n}$.
\end{problem}

If we put $A_{n}=\sum_{k=1}^{n}a_{k}$, it becomes clear that this equation
involves first and second differences of $A_{n}$. Thus if we use the
dictionary with the function $F$ corresponding to $A_{n}$ and $f$
corresponding to $a_{n}$ or $F^{\prime }$, then the given difference
equation and its continuous analogue are 
\[
a_{n+1}=a_{n}+\frac{A_{n}}{a_{n}},F^{\prime \prime }(x)F^{\prime }(x)=F(x). 
\]%
In the continuous equation, multiplying by $F^{\prime }$ and integrating, we
are led, successively, to the equations 
\[
\frac{f^{3}(x)}{3}=\frac{F^{2}(x)}{2},F^{1/3}(x)=\frac{1}{3}(\frac{3}{2}%
)^{1/3}x,F(x)=\frac{x^{3}}{18},f(x)=\frac{x^{2}}{6}. 
\]%
Thus we guess that $a_{n}$ $\sim $ $\frac{n^{2}}{6}$, and we have to start
by obtaining a relationship between $2a_{n}^{3}$ $\ $and $3A_{n}^{2}$, as
suggested by the continuous case.

Let us see what can be done from very simple considerations. We have,
successively, $a_{n}\geq 1$, $a_{n+1}-a_{n}\geq 1,a_{n+1}\geq n+1,A_{n}\geq 
\frac{n(n+1)}{2}>\frac{n^{2}}{2}$. These inequalities, are quite far from
the expected result, but could be useful in the analysis, notably in
achieving a certain necessary matching of the indices as we shall see
presently. But at least they tell us that both $a_{n}$ and $A_{n}$ tend to $%
\infty $ as $n\rightarrow \infty ,$ and, in addition, that 
\[
\frac{nA_{n-1}}{A_{n+1}^{2}}\leq \frac{2}{n},n\geq 2. 
\]%
We shall use the two identities: 
\[
a_{n+1}^{3}-a_{n}^{3}=3a_{n}A_{n}+3\frac{A_{n}^{2}}{a_{n}}+\frac{A_{n}^{3}}{%
a_{n}^{3}},a_{n+1}^{3}-a_{n}^{3}=3a_{n+1}A_{n}+\frac{A_{n}^{3}}{a_{n}^{3}}. 
\]%
In the first identity, the positivity of all terms involved gives us%
\[
a_{n+1}^{3}-a_{n}^{3}\geq 3a_{n}A_{n},2(a_{n+1}^{3}-a_{n}^{3})\geq
3a_{n}A_{n}+3a_{n}A_{n-1}=3(A_{n}^{2}-A_{n-1}^{2}) 
\]%
which, upon summing, with $A_{0}=0$, leads to 
\[
2a_{n+1}^{3}\geq 3A_{n}^{2}+2a_{1}^{3}\geq 3A_{n}^{2}. 
\]%
In a similar manner, the second identity, leads to%
\[
2(a_{n+1}^{3}-a_{n}^{3})\leq 3a_{n+1}A_{n}+3a_{n+1}A_{n+1}+2\frac{A_{n}^{3}}{%
a_{n}^{3}}=3(A_{n+1}^{2}-A_{n}^{2})+2\frac{A_{n}^{3}}{a_{n}^{3}} 
\]%
and 
\[
2a_{n+1}^{3}-2a_{1}^{3}\leq 3A_{n+1}^{2}-3A_{1}^{2}+2\sum_{k=1}^{n}\frac{%
A_{k}^{3}}{a_{k}^{3}}. 
\]

\bigskip

Since $A_{n}^{1/3}\geq 1$, for $n\geq 1$, 
\[
\left( \frac{A_{n}}{a_{n}}\right) ^{3}=\left( 1+\frac{A_{n-1}}{a_{n}}\right)
^{3}\leq \left( 1+(\frac{2}{3})^{1/3}A_{n-1}^{1/3}\right) ^{3}\leq
8A_{n-1},n\geq 2. 
\]%
It follows then that 
\[
2a_{n+1}^{3}-2\leq 3A_{n+1}^{2}+16\sum_{k=2}^{n}A_{k-1}\leq 3A_{n+1}^{2}+%
\frac{32}{n}A_{n+1}^{2},n\geq 2. 
\]%
Thus we have, successively, 
\[
\lim \sup_{n\rightarrow \infty }\frac{2a_{n}^{3}}{3A_{n}^{2}}\leq
1,\lim_{n\rightarrow \infty }\frac{a_{n}}{A_{n}}=0,\lim_{n\rightarrow \infty
}\frac{A_{n+1}}{A_{n}}=1,\liminf_{n\rightarrow \infty }\frac{2a_{n}^{3}}{%
3A_{n}^{2}}\geq 1\text{,}\lim_{n\rightarrow \infty }\frac{2a_{n}^{3}}{%
3A_{n}^{2}}=1, 
\]%
and the first asymptotic relation is proved. It remains to obtain the
behaviour of $A_{n}$ and then $a_{n}.$To this end we cast our first
asymptotic result in the two equivalent forms%
\[
\lim_{n\rightarrow \infty }\frac{A_{n+1}-A_{n}}{A_{n+1}^{2/3}}=\left( \frac{3%
}{2}\right) ^{1/3},\lim_{n\rightarrow \infty }\frac{A_{n+1}-A_{n}}{%
A_{n}^{2/3}}=\left( \frac{3}{2}\right) ^{1/3}. 
\]%
Next, a use of the inequalities%
\[
\frac{1}{3}\cdot \frac{A_{n+1}-A_{n}}{A_{n+1}^{2/3}}\leq \frac{1}{3}%
\int_{A_{n}}^{A_{n+1}}x^{-2/3}dx\leq \frac{1}{3}\cdot \frac{A_{n+1}-A_{n}}{%
A_{n}^{2/3}}, 
\]%
along with the abelian result mentioned previously, readily leads to 
\[
\lim_{n\rightarrow \infty }\frac{A_{n+1}^{1/3}-1}{n}=\frac{1}{3}\cdot \left( 
\frac{3}{2}\right) ^{1/3}. 
\]%
Finally we have 
\[
A_{n}\text{ }\sim \text{ }\frac{1}{18}n^{3},a_{n}\text{ }\sim \text{ }\frac{1%
}{6}n^{2},n\rightarrow \infty . 
\]

\section{Tauberian results}

Suppose $T$ is a transform, whose exact form need not concern us here, that
takes sequences into sequences, and we are given the existence of $%
\lim_{n\rightarrow \infty }Ta_{n}$. It is often necessary to find the
asymptotic behaviour of the sequence $a_{n}$ . One of the difficulties
encountered in such problems is that it appears that very little is given,
and one doesn't have a clue as to how to start an attack on such a problem.
But if the problem admits of a translation into a continuous analogue, then
a line of attack might be suggested by the corresponding solution of the
continuous problem. Before we present an illustration of this, it will be
useful to have at hand the following well-known abelian result on sequences.

\textit{If }$a_{n}$\textit{\ is a real or complex sequence, and }$%
a_{n}\rightarrow a$\textit{, as }$n\rightarrow \infty $\textit{, then }%
\[
\frac{a_{1}+\cdot \cdot \cdot +a_{n}}{n}\rightarrow a,n\rightarrow \infty . 
\]%
Our next example illustrates the use of the method in a Tauberian problem.

\begin{problem}
\textit{If }$a_{n}$\textit{\ is a sequence of positive real numbers, and }%
\[
\lim_{n\rightarrow \infty }a_{n}\sum_{k=1}^{n}a_{k}^{2}=1\text{,} 
\]%
\textit{determine the order of growth of }$a_{n}$\textit{.\ \ \ \ \ }
\end{problem}

We employ the dictionary to move from the given problem to its continuous
analogue. So we have a positive function $f$ satisfying%
\[
\lim_{x\rightarrow \infty }f(x)\int_{1}^{x}f^{2}(t)dt=1\text{.} 
\]%
If we recast this in the equivalent form 
\[
\lim_{x\rightarrow \infty }f^{2}(x)\left( \int_{1}^{x}f^{2}(t)dt\right)
^{2}=1, 
\]%
then with the introduction of $F(x)=$\ \ \ $\int_{1}^{x}f^{2}(t)dt$\ $,$ we
have that $\lim_{x\rightarrow \infty }F^{\prime }(x)F^{2}(x)=1$, which, by
L'Hospital's rule \cite{Rudin}, implies that $\lim_{x\rightarrow \infty }%
\frac{F^{3}(x)}{3x}=1$, so that $\lim_{x\rightarrow \infty }\sqrt[3]{3x}%
f(x)=1$. So we guess that our sequence $a_{n}$ satisfies $\lim_{n\rightarrow
\infty }$\ \ \ \ $\sqrt[3]{3n}$\ $a_{n}=1$, i.e. the sequence decays to zero
like $\frac{1}{\ \sqrt[3]{3n}}$.\ The continuous analogue also suggests a
line of attack: introduce the sums $A_{n}=\sum_{k=1}^{n}a_{k}^{2}$, and look
at $A_{n}^{3}$.\ \ \ Before proceeding any further, let us note that, since
the method anticipates that the sequence $a_{n}$ decays to zero like a power
of $n$, we should see if we can first establish the weaker result that the
sequence $a_{n}$ does actually tend to $0.$ To this end note that, if the
series $\sum_{k=1}^{\infty }a_{k}^{2}$ converges then $\lim_{n\rightarrow
\infty }a_{n}A_{n}=0$ contrary to the given hypothesis. Thus the series
diverges, and then the sequence $A_{n}$ being monotone increasing and
unbounded must tend to $\infty $, and so $\lim_{n\rightarrow \infty
}a_{n}=\lim_{n\rightarrow \infty }a_{n}A_{n}\cdot \frac{1}{A_{n}}=0$. So $%
a_{n}\rightarrow 0$, and we can proceed to investigate its order of decay to 
$0$. Motivated by the continuous analogue we compute%
\[
A_{n+1}^{3}-A_{n}^{3}=a_{n+1}^{2}\left(
A_{n+1}^{2}+A_{n+1}A_{n}+A_{n}^{2}\right) . 
\]%
We are given that $\lim_{n\rightarrow \infty }a_{n}A_{n}=1$, so that we only
need to find the limit 
\[
\lim_{n\rightarrow \infty }a_{n+1}A_{n}. 
\]%
Since $a_{n+1}A_{n}=a_{n+1}A_{n+1}-a_{n+1}^{3}$\ $,$ we also obtain that \ \
\ $\lim_{n\rightarrow \infty }a_{n+1}A_{n}=1$, so that $\lim_{n\rightarrow
\infty }$\ $\left( \ \ A_{n+1}^{3}-A_{n}^{3}\ \right) =3$. But then it
follows, by the abelian result mentioned earlier, that 
\[
\lim_{n\rightarrow \infty }\frac{A_{n+1}^{3}-A_{1}^{3}}{n}%
=\lim_{n\rightarrow \infty }\frac{\sum_{k=1}^{n}\left( \ \
A_{k+1}^{3}-A_{k}^{3}\ \right) }{n}=3. 
\]%
This gives $\lim_{n\rightarrow \infty }\frac{A_{n}^{3}}{n}=3$, and we
conclude that \ \ \ \ $\lim_{n\rightarrow \infty }$\ \ \ \ $\sqrt[3]{3n}$\ $%
a_{n}=1$, as expected.

The reader will have no difficulty in formulating a generalization of this
if we are given that $a_{n}>0$, and 
\[
\lim_{n\rightarrow \infty }a_{n}^{p}\sum_{k=1}^{n}a_{k}^{q}=1, 
\]%
where $p,q$ are positive integers.

Well, whether $p,q$ are positive integers or just positive real numbers, the
continuous analogue will be the problem

\[
\lim_{x\rightarrow \infty }f^{p}(x)\int_{1}^{x}f^{q}(t)dt=1, 
\]%
which, when recast in the equivalent form \ \ \ \ \ \ 

\[
\lim_{x\rightarrow \infty }F^{\prime
}(x)F^{q/p}(x)=1,F(x)=\int_{1}^{x}f^{q}(t)dt, 
\]%
tells us that $\lim_{x\rightarrow \infty }$\ $\frac{F^{\frac{q}{p}+1}(x)}{%
\left( \frac{q}{p}+1\right) x}=1$, leads us to guess the appropriate decay
of the sequence $a_{n}$, and suggests that we start by putting $%
A_{n}=\sum_{k=1}^{n}a_{k}^{q}$, and compute first differences of the $A_{n}^{%
\frac{q}{p}+1}$.\ But here, when $\frac{q}{p}+1$ is not a positive
integer,there is no simple factorization of the difference. To proceed, we
have to appeal to the continuous setting. Indeed, using that the derivative
of $t^{\frac{q}{p}+1}$ is $\left( \frac{q}{p}+1\right) t^{\frac{q}{p}},$ we
can write \ \ \ \ \ \ \ \ \ \ \ \ \ \ \ \ \ \ \ \ \ \ \ \ \ \ \ \ \ \ \ \ \
\ \ \ \ \ \ \ \ \ \ \ \ \ \ \ \ \ \ \ \ \ \ \ \ \ \ \ \ \ \ \ \ \ \ \ \ \ \
\ \ \ \ \ \ \ \ \ \ \ \ \ \ \ \ \ \ \ \ \ \ \ \ \ \ \ \ \ \ \ \ \ \ \ \ \ \
\ \ \ \ \ \ \ \ \ \ \ \ \ \ \ \ \ \ \ \ \ \ \ \ \ \ \ \ \ \ \ \ \ \ \ \ \ \
\ \ \ \ \ \ \ \ \ \ \ \ \ \ \ \ \ \ \ \ \ \ \ \ \ \ \ \ \ \ \ \ \ \ \ \ \ \
\ \ \ \ \ \ \ \ \ \ \ \ \ \ \ \ \ \ \ \ \ \ \ \ \ \ \ \ \ \ \ \ \ \ \ \ \ \
\ \ \ \ \ \ \ \ \ \ \ \ \ \ \ \ \ \ \ \ \ \ \ \ \ \ \ \ \ \ \ \ \ \ \ \ \ \
\ \ \ \ \ \ \ \ \ \ \ \ \ \ \ \ \ \ \ \ \ \ \ \ \ \ \ \ \ \ \ \ \ \ \ \ \ \
\ \ \ \ \ \ \ \ \ \ \ \ \ \ \ \ \ \ \ \ \ \ \ \ \ \ \ \ \ \ \ \ \ \ \ \ \ \
\ \ \ \ \ \ \ \ \ \ \ \ \ \ \ \ \ \ \ \ \ \ \ \ \ \ \ \ \ \ \ \ \ \ \ \ \ \
\ \ \ \ \ \ \ \ \ \ \ \ \ \ \ \ \ \ \ \ \ \ \ \ \ \ \ \ \ \ \ \ \ \ \ \ \ \
\ \ \ \ \ \ \ \ \ \ \ \ \ \ \ \ \ \ \ \ \ \ \ \ \ \ \ \ \ \ \ \ \ \ \ \ \ \
\ \ \ \ \ \ \ \ \ \ \ \ \ \ \ \ \ \ \ \ \ \ \ \ \ \ \ \ \ \ \ \ \ \ \ \ \ \
\ \ \ \ \ \ \ \ \ \ \ \ \ \ \ \ \ \ \ \ \ \ \ \ \ \ \ \ \ \ \ \ \ \ \ \ \ \
\ \ \ \ \ \ \ \ \ \ \ \ \ \ \ \ \ \ \ \ \ \ \ \ \ \ \ \ \ \ \ \ \ \ \ \ \ \
\ \ \ \ \ \ \ \ \ \ \ \ \ \ \ \ \ \ \ \ \ \ \ \ \ \ \ \ \ \ \ \ \ \ \ \ \ \
\ \ \ \ \ \ \ \ \ \ \ \ \ \ \ \ \ \ \ \ \ \ \ \ \ \ \ \ \ \ \ \ \ \ \ \ \ \
\ \ \ \ \ \ \ \ \ \ \ \ \ \ \ \ \ \ \ \ \ \ \ \ \ \ \ \ \ \ \ \ \ \ \ \ \ \
\ \ \ \ \ \ \ \ \ \ \ \ \ \ \ \ \ \ \ \ \ \ \ \ \ \ \ \ \ \ \ \ \ \ \ \ \ \
\ \ \ \ \ \ \ \ \ \ \ \ \ \ \ \ \ \ \ \ \ \ \ \ \ \ \ \ \ \ \ \ \ \ \ \ \ \
\ \ \ \ \ \ \ \ \ \ \ \ \ \ \ \ \ \ \ \ \ \ \ \ \ \ \ \ \ \ \ \ \ \ \ \ \ \
\ \ \ \ \ \ \ \ \ \ \ \ \ \ \ \ \ \ \ \ \ \ \ \ \ \ \ \ \ \ \ \ \ \ \ \ \ \
\ \ \ \ \ \ \ \ \ \ \ \ \ \ \ \ \ \ \ \ \ \ \ \ \ \ \ \ \ \ \ \ \ \ \ \ \ \
\ \ \ \ \ \ \ \ \ \ \ \ \ \ \ \ \ \ \ \ \ \ \ \ \ \ \ \ \ \ \ \ \ \ \ \ \ \
\ \ \ \ \ \ \ \ \ \ \ \ \ \ \ \ \ \ \ \ \ \ \ \ \ \ \ \ \ \ \ \ \ \ \ \ \ \
\ \ \ \ \ \ \ \ \ \ \ \ \ \ \ \ \ \ \ \ \ \ \ \ \ \ \ \ \ \ \ \ \ \ \ \ \ \
\ \ \ \ \ \ \ \ \ \ \ \ \ \ \ \ \ \ \ \ \ \ \ \ \ \ \ \ \ \ \ \ \ \ \ \ \ \
\ \ \ \ \ \ \ \ \ \ \ \ \ \ \ \ \ \ \ \ \ \ \ \ \ \ \ \ \ \ \ \ \ \ \ \ \ \
\ \ \ \ \ \ \ \ \ \ \ \ \ \ \ \ \ \ \ \ \ \ \ \ \ \ \ \ \ \ \ \ \ \ \ \ \ \
\ \ \ \ \ \ \ \ \ \ \ \ \ \ \ \ \ \ \ \ \ \ \ \ \ \ \ \ \ \ \ \ \ \ \ \ \ \
\ \ \ \ \ \ \ \ \ \ \ \ \ \ \ \ \ \ \ \ \ \ \ \ \ \ \ \ \ \ \ \ \ \ \ \ \ \
\ \ \ \ \ \ \ \ \ \ \ \ \ \ \ \ \ \ \ \ \ \ \ \ \ \ \ \ \ \ \ \ \ \ \ \ \ \
\ \ \ \ \ \ \ \ \ \ \ \ \ \ \ \ \ \ \ \ \ \ \ \ \ \ \ \ \ \ \ \ \ \ \ \ \ \
\ \ \ \ \ \ \ \ \ \ \ \ \ \ \ \ \ \ \ \ \ \ \ \ \ \ \ \ \ \ \ \ \ \ \ \ \ \
\ \ \ \ \ \ \ \ \ \ \ \ \ \ \ \ \ \ \ \ \ \ \ \ \ \ \ \ \ \ \ \ \ \ \ \ \ \
\ \ \ \ \ \ \ \ \ \ \ \ \ \ \ \ \ \ \ \ \ \ \ \ \ \ \ \ \ \ \ \ \ \ \ \ \ \
\ \ \ \ \ \ \ \ \ \ \ \ \ \ \ \ \ \ \ \ \ \ \ \ \ \ \ \ \ \ \ \ \ \ \ \ \ \
\ \ \ \ \ \ \ \ \ \ \ \ \ \ \ \ \ \ \ \ \ \ \ \ \ \ \ \ \ \ \ \ \ \ \ \ \ \
\ \ \ \ \ \ \ \ \ \ \ \ \ \ \ \ \ \ \ \ \ \ \ \ \ \ \ \ \ \ \ \ \ \ \ \ \ \
\ \ \ \ \ \ \ \ \ \ \ \ \ \ \ \ \ \ \ \ \ \ \ \ \ \ \ \ \ \ \ \ \ \ \ \ \ \
\ \ \ \ \ \ \ \ \ \ \ \ \ \ \ \ \ \ \ \ \ \ \ \ \ \ \ \ \ \ \ \ \ \ \ \ \ \
\ \ \ \ \ \ \ \ \ \ \ \ \ \ \ \ \ \ \ \ \ \ \ \ \ \ \ \ \ \ \ \ \ \ \ \ \ \
\ \ \ \ \ \ \ \ \ \ \ \ \ \ \ \ \ \ \ \ \ \ \ \ \ \ \ \ \ \ \ \ \ \ \ \ \ \
\ \ \ \ \ \ \ \ \ \ \ \ \ \ \ \ \ \ \ \ \ \ \ \ \ \ \ \ \ \ \ \ \ \ \ \ \ \
\ \ \ \ \ \ \ \ \ \ \ \ \ \ \ \ \ \ \ \ \ \ \ \ \ \ \ \ \ \ \ \ \ \ \ \ \ \
\ \ \ \ \ \ \ \ \ \ \ \ \ \ \ \ \ \ \ \ \ \ \ \ \ \ \ \ \ \ \ \ \ \ \ \ \ \
\ \ \ \ \ \ \ \ \ \ \ \ \ \ \ \ \ \ \ \ \ \ \ \ \ \ \ \ \ \ \ \ \ \ \ \ \ \
\ \ \ \ \ \ \ \ \ \ \ \ \ \ \ \ \ \ \ \ \ \ \ \ \ \ \ \ \ \ \ \ \ \ \ \ \ \
\ \ \ \ \ \ \ \ \ \ \ \ \ \ \ \ \ \ \ \ \ \ \ \ \ \ \ 
\[
A_{n+1}^{\frac{q}{p}+1}-A_{n}^{\frac{q}{p}+1}=\int_{A_{n}}^{A_{n+1}}\left( 
\frac{q}{p}+1\right) t^{\frac{q}{p}}dt. 
\]%
Once we have this we obtain immediately a double inequality that replaces
the identity in the case of cubes:%
\[
\left( \frac{q}{p}+1\right) A_{n}^{\frac{q}{p}}a_{n+1}^{q}\leq A_{n+1}^{%
\frac{q}{p}+1}-A_{n}^{\frac{q}{p}+1}\leq \left( \frac{q}{p}+1\right)
A_{n+1}^{\frac{q}{p}}(A_{n+1}-A_{n})=\left( \frac{q}{p}+1\right) A_{n+1}^{%
\frac{q}{p}}a_{n+1}^{q}. 
\]%
On the right-hand side we have $A_{n+1}^{\frac{q}{p}}a_{n+1}^{q}=\left(
A_{n+1}a_{n+1}^{p}\right) ^{\frac{q}{p}}\rightarrow 1$ by hypothesis. On the
left-hand side we have $A_{n}^{\frac{q}{p}}a_{n+1}^{q}=\left(
A_{n}a_{n+1}^{p}\right) ^{\frac{q}{p}}=\left(
A_{n+1}a_{n+1}^{p}-a_{n+1}^{p+1}\right) ^{\frac{q}{p}}\rightarrow 1$. We can
thus proceed as before and obtain that 
\[
\lim_{n\rightarrow \infty }\frac{A_{n}^{\frac{q}{p}+1}}{n}=\left( \frac{q}{p}%
+1\right) \text{,} 
\]%
thereby obtaining the expected decay of $A_{n}$ , and so of $a_{n}$.

\section{Coupled systems}

For problems involving a complex sequence $z_{n}=a_{n}+ib_{n}$, the first
thing that comes to mind is to separate into real and imaginary parts, to
obtain a system of two real problems. As a simple example, the problem of
the asymptotic behaviour of 
\[
z_{n+1}=z_{n}+\frac{i}{nz_{n}},z_{1}=1+i, 
\]%
is equivalent to that of the coupled real system%
\[
a_{n+1}=a_{n}+\frac{b_{n}}{n(a_{n}^{2}+b_{n}^{2})},b_{n+1}=b_{n}+\frac{a_{n}%
}{n(a_{n}^{2}+b_{n}^{2})},a_{1}=b_{1}=1. 
\]%
Such systems are expected to be much more involved. But the method of the
continuous analogue, can still be utilized as we hope to demonstrate in the
result that follows. Exploration of the many possible generalizations is
left to the reader.

Let $a_{n}$ and $b_{n}$ be the two sequences defined by the coupled system%
\[
a_{n+1}-a_{n}=\frac{1}{b_{n}^{2}},b_{n+1}-b_{n}=\frac{1}{a_{n}^{2}}%
,a_{1}>0,b_{1}>0. 
\]%
Our purpose here is to study the asymptotic behavior of each of these two
sequences.

The continuous analogue, ignoring the initial conditions, consists of two
positive functions satisfying the equations%
\[
f^{\prime }(t)=\frac{1}{g^{2}(t)},g^{\prime }(t)=\frac{1}{f^{2}(t)}, 
\]%
from which the one particular solution $f(t)=g(t)=(3t)^{1/3}$ is easily
derived. In particular $f^{3}(t)g^{3}(t)=9t^{2}.$ So we guess that the
sequences satisfy the asymptotic relations $a_{n}\thicksim
(3n)^{1/3},b_{n}\thicksim (3n)^{1/3},n\rightarrow \infty .$ We also see that
we might start with $a_{n}^{3}$ and $b_{n}^{3}$. Any inequality between
these two will lead to a decoupling of the system and allow us to work
separately on each sequence. We start by some straightforward observations:

\textit{The sequences }$a_{n}$\textit{\ and }$b_{n}$\textit{\ are positive
increasing sequences. Denote their limits by }$a$\textit{\ and }$b$\textit{\
respectively. Then }$0<a\leq \infty ,$\textit{\ and }$0<b\leq \infty .$

\textit{If }$0<a<\infty $\textit{, then }$\frac{b_{n}}{n}\rightarrow \frac{1%
}{a^{2}},$\textit{\ and \ }$a-a_{n}\thicksim \frac{a^{4}}{n}$\textit{, as }$%
n\rightarrow \infty .$\textit{\ If }$0<b<\infty $\textit{, then }$\frac{a_{n}%
}{n}\rightarrow \frac{1}{b^{2}},$\textit{\ and }$b-b_{n}\thicksim \frac{b^{4}%
}{n}$\textit{, as }$n\rightarrow \infty .$\textit{\ In particular, the
limits }$a$\textit{\ and }$b$\textit{\ cannot both be finite.}

If both $a$ and $b$ are infinite, then the asymptotic behavior of $a_{n}$
and $b_{n}$ is described in the following result:

\begin{theorem}
Let $a_{n}$ and $b_{n}$ be the two sequences defined above, let $a$ and $b$
be their limits, and assume that $a=b=\infty .$ Put 
\[
\alpha =\limsup_{n\rightarrow \infty }n^{-\frac{1}{3}}a_{n},\alpha ^{\prime
}=\liminf_{n\rightarrow \infty }n^{-\frac{1}{3}}a_{n},\beta
=\limsup_{n\rightarrow \infty }n^{-\frac{1}{3}}b_{n},\beta ^{\prime
}=\liminf_{n\rightarrow \infty }n^{-\frac{1}{3}}b_{n}. 
\]%
Then $\alpha ^{\prime }=\alpha =\beta =\beta ^{\prime }=3^{1/3}$, and 
\[
a_{n}\thicksim (3n)^{1/3},b_{n}\thicksim (3n)^{1/3},n\rightarrow \infty . 
\]
\end{theorem}

\begin{proof}
As suggested by the continuous case we consider the cubic powers of the two
sequences with a view to obtain an inequality for their product.\bigskip\
The monotonicity of $a_{n}$ and $b_{n}$, gives the elementary
inequalities\bigskip 
\[
a_{n+1}^{3}-a_{n}^{3}\geq 3a_{n}^{2}(a_{n+1}-a_{n})=3\frac{a_{n}^{2}}{%
b_{n}^{2}},b_{n+1}^{3}-b_{n}^{3}\geq 3b_{n}^{2}(b_{n+1}-b_{n})=3\frac{%
b_{n}^{2}}{a_{n}^{2}}. 
\]%
\bigskip Upon summation and multiplication, these in turn imply 
\[
\left( a_{n+1}^{3}-a_{1}^{3}\right) \left( b_{n+1}^{3}-b_{1}^{3}\right) \geq
9\left( \sum_{k=1}^{n}\frac{a_{k}^{2}}{b_{k}^{2}}\right) \left(
\sum_{k=1}^{n}\frac{b_{k}^{2}}{a_{k}^{2}}\right) \geq 9n^{2},n\geq 1, 
\]%
giving the \ "uncertainty" inequality 
\[
a_{n+1}^{3}\cdot b_{n+1}^{3}\geq 9n^{2},n\geq 1. 
\]%
Of course, the matching is not perfect, but this is sort of in the nature of
such problems. Rewrite the uncertainty inequality in the equivalent form \ \ 
\[
a_{n+1}^{-2}\cdot b_{n+1}^{-2}\leq 9^{-\frac{2}{3}}n^{-\frac{4}{3}},n\geq 1, 
\]%
and use it to get\ 
\[
a_{n+1}^{-1}-a_{n+2}^{-1}=\int_{a_{n+1}}^{a_{n+2}}t^{-2}dt\leq
a_{n+1}^{-2}(a_{n+2}-a_{n+1})\leq 9^{-\frac{2}{3}}n^{-\frac{4}{3}}. 
\]%
Now sum this up from $n\geq 2$, to $m\geq n+2$ to obtain%
\[
a_{n+1}^{-1}-a_{m}^{-1}\leq 9^{-\frac{2}{3}}\sum_{k=n}^{m-2}k^{-\frac{4}{3}%
}. 
\]%
\bigskip \bigskip We now let $m\rightarrow \infty $, and use the hypothesis
that $a_{m}\rightarrow \infty $ as $m\rightarrow \infty $ to obtain\bigskip 
\[
a_{n+1}^{-1}\leq 9^{-\frac{2}{3}}\sum_{k=n}^{\infty }k^{-\frac{4}{3}}\leq
3\cdot 9^{-\frac{2}{3}}(n-1)^{-\frac{1}{3}}=\frac{1}{\sqrt[3]{3(n-1)}},n\geq
2. 
\]%
A similar result holds for $b_{n}$ and we conclude that 
\[
\liminf_{n\rightarrow \infty }n^{-\frac{1}{3}}a_{n}\geq
1,\liminf_{n\rightarrow \infty }n^{-\frac{1}{3}}b_{n}\geq 1. 
\]%
To get inequalities in the opposite direction, we return to the defining
equations of the sequences and incorporate in them these new inequalities.
Thus, for the first sequence, we get 
\[
a_{n+1}-a_{n}\leq \frac{1}{\left( 3(n-1)\right) ^{2/3}},n\geq 2\text{, } 
\]%
which upon summation and estimation of the resulting sum with the
corresponding integral, leads to 
\[
a_{n+1}\leq a_{2}+c+\left( 3(n-1)\right) ^{1/3},n\leq 2, 
\]%
where $c$ is a constant. It follows that $\limsup_{n\rightarrow \infty }n^{-%
\frac{1}{3}}a_{n}\leq 1$, and hence the limit exists and equals $1$. The
same applies to the sequence $b_{n}.$ This completes the proof of the
theorem.
\end{proof}

\section{Comparison of sequences}

The constant sequence $x_{n}=1$, and the sequence $y_{n}=\sin ^{2}n$, are
both positive and bounded, with $y_{n}\leq x_{n}$, for all $n$. Is there a
way of quantifying, how much smaller is $y_{n}?$ We propose to answer this
question as follows: introduce two sequences $a_{n},b_{n}$ defined by 
\[
a_{n+1}=a_{n}+\frac{x_{n}}{a_{n}},b_{n+1}=b_{n}+\frac{y_{n}}{b_{n}}%
,a_{1}=b_{1}=1. 
\]%
Then our method here may be used to obtain the asymptotic relations%
\[
a_{n}\sim \sqrt{2n},b_{n}\sim \sqrt{n},n\rightarrow \infty . 
\]%
Thus the "excess" in magnitude of $x_{n}$ over $y_{n}$ is quantified by the
constant $\sqrt{2}$ present in these asymptotic relations.

\section{Beyond the first term}

So far all illustrations involved what is called the first term describing
the asymptotic behavior of a given sequence. It is often desirable, and very
useful, to find what is called the second term in the asmptotic expansion of
the sequence. Of course we have not defined what an asymptotic expansion is,
but we hope it will be clear from the next discussion what is intended. Let
us go directly to an example, which we think , is quite suitable for our
purpose.

\begin{example}
The sequence $x_{n}$ is defined by the recurrence%
\[
x_{n+1}=x_{n}-x_{n}^{2},0<x_{1}<1. 
\]%
Obtain the asymptotic behaviour of $x_{n}$.
\end{example}

First of all, it is easy to see that this sequence is positive and
decreasing. That is $0<x_{n+1}<x_{n}<1$. So it is convergent, and it is
immediate that its limit is $0$.The corresponding continuous equation, $%
f^{\prime }(t)=-f^{2}(t)$ has $f(t)=t^{-1}$ as one solution and we guess
that $x_{n}$ $\sim \frac{1}{n}$, or, put differently, that $%
nx_{n}-1\rightarrow 0$ as $n\rightarrow \infty .$ If this can be proved,
then we can go one step further and ask about the rate at which $nx_{n}-1$
decays to $0$. We trust that the reader will be able to show that $x_{n}\sim 
\frac{1}{n}$, and, by induction, that $nx_{n}<1$ for all $n.$So let us
attend to the question of the rate of decay of $nx_{n}-1$. In order to apply
our method, we need to put $a_{n}=nx_{n}$, and find a new equation satisfied
by $a_{n}$ and then apply the continuous method to this equation. Perhaps
the simplest way to do this is to multiply the given recurrence equation for 
$x_{n}$ by $n+1$ and then write everything in terms of $a_{n}.$ We are thus
led to the problem of the asymptotic behaviour of $a_{n}$ where 
\[
a_{n+1}=a_{n}+\frac{a_{n}}{n}-\frac{(n+1)a_{n}^{2}}{n^{2}},0<a_{1}<1. 
\]%
The continuous analogue is a positive function $f$ satisfying 
\[
f^{\prime }(t)=\frac{f(t)}{t}-(t+1)\left( \frac{f(t)}{t}\right) ^{2}. 
\]%
The introduction of $g(t)=\frac{f(t)}{t}$, leads to the differential equation%
\[
g^{\prime }(t)=-\frac{t+1}{t}g^{2}(t), 
\]%
which is solved by $g(t)=\frac{1}{t+\ln t}$, so that $f(t)=\frac{t}{t+\ln t}$%
. Thus $1-f(t)=\frac{\ln t}{t+\ln t}$, and we guess that 
\[
1-nx_{n}\sim \frac{\ln n}{n+\ln n}\sim \frac{\ln n}{n},n\rightarrow \infty . 
\]

\end{document}